\documentclass[a4paper,12pt]{amsart}
\usepackage{amsmath,amssymb,amsthm}

\theoremstyle{plain}
\newtheorem{thm}{Theorem}[section]
\newtheorem{lem}[thm]{Lemma}
\newtheorem{prop}[thm]{Proposition}
\newtheorem{cor}[thm]{Corollary}
\newtheorem{conj}[thm]{Conjecture}

\theoremstyle{definition}
\newtheorem{defn}[thm]{Definition}

\newtheorem{rem}[thm]{Remark}
\newtheorem{ex}[thm]{Example}

\numberwithin{equation}{section}

\newcommand{\Q}{\mathbb{Q}}
\newcommand{\R}{\mathbb{R}}

\newcommand{\bk}{\mathbf{k}}
\newcommand{\frA}{\mathfrak{A}}

\newcommand{\frh}{\mathfrak{h}}
\newcommand{\frz}{\mathfrak{z}}

\newcommand{\hp}[1]{\overset{#1}{*}}
\newcommand{\odd}{\mathrm{odd}}

\DeclareMathOperator{\Con}{\mathrm{Con}}
\DeclareMathOperator{\Ker}{\mathrm{Ker}}

\author{Shuji Yamamoto}
\address{
JSPS Research Fellow \\ 
Graduate School of Mathematical Sciences\\ 
The University of Tokyo\\ 
3-8-1 Komaba, Meguro, Tokyo, 153-8914 Japan.}
\email{yamashu@ms.u-tokyo.ac.jp}
\thanks{This work was supported by Grant-in-Aid 
for JSPS Fellows 21$\cdot$5093}

\title{Interpolation of multiple zeta and zeta-star values}
\date{}
\keywords{multiple zeta values, multiple zeta-star values}
\subjclass[2000]{Primary 11M32, Secondary 16W99}

\begin{document}

\begin{abstract}
We define polynomials of one variable $t$ whose values at $t=0$ and $1$ are 
the multiple zeta values and the multiple zeta-star values, respectively. 
We give an application to the two-one conjecture of Ohno-Zudilin, 
and also prove the cyclic sum formula for these polynomials. 
\end{abstract}

\maketitle 

\section{Introduction}
For integers $k_1\geq 2$ and $k_2,\ldots,k_n\geq 1$, 
the multiple zeta value and the multiple zeta-star value 
(MZV and MZSV for short) are defined respectively as follows: 
\begin{align*}
\zeta(k_1,\ldots,k_n)&=\sum_{m_1>\cdots>m_n>0}
\frac{1}{m_1^{k_1}\cdots m_n^{k_n}}, \\
\zeta^\star(k_1,\ldots,k_n)&=\sum_{m_1\geq\cdots\geq m_n>0}
\frac{1}{m_1^{k_1}\cdots m_n^{k_n}}. 
\end{align*}
It is easy to see that an MZSV can be expressed as a linear combination 
of MZVs, and vice versa. Indeed, 
\begin{equation}\label{eq:MZSV by MZV}
\zeta^\star(k_1,\ldots,k_n)=\sum_{\mathbf{p}}\zeta(\mathbf{p}), 
\end{equation}
where $\mathbf{p}$ runs over all indices of the form 
\[\mathbf{p}=(k_1\square k_2\square\cdots\square k_n)\]
in which each $\square$ is filled by the comma $,$ or the plus $+$. 
If we denote by $\sigma(\mathbf{p})$ the number of $+$ 
used in $\mathbf{p}$, we also have 
\begin{equation}\label{eq:MZV by MZSV}
\zeta(k_1,\ldots,k_n)=\sum_{\mathbf{p}}
(-1)^{\sigma(\mathbf{p})}\zeta^\star(\mathbf{p}). 
\end{equation}
In this paper, we introduce the polynomial 
\begin{equation}\label{eq:poly}
\zeta^t(k_1,\ldots,k_n)=\sum_{\mathbf{p}}
t^{\sigma(\mathbf{p})}\zeta(\mathbf{p})
\end{equation}
of one variable $t$. 
Note that $\zeta^0(\bk)=\zeta(\bk)$ and $\zeta^1(\bk)=\zeta^\star(\bk)$, 
i.e., this polynomial interpolates MZV and MZSV. 

There are many $\Q$-linear relations among MZVs, 
and some (families of) relations are named after their origin 
or their form. For example, the relation 
\begin{equation}\label{eq:SF MZV}
\sum_{\substack{k_1\geq 2,k_2,\ldots,k_n\geq 1,\\ k_1+\cdots+k_n=k}}
\zeta(\bk)=\zeta(k) \qquad (k>n\geq 1)
\end{equation}
is called the sum formula, and the relation 
\[\begin{split}
&\zeta(a_1+1,\{1\}^{b_1-1},\ldots,a_s+1,\{1\}^{b_s-1})\\
&=\zeta(b_s+1,\{1\}^{a_s-1},\ldots,b_1+1,\{1\}^{a_1-1})
\qquad (a_1,b_1,\ldots,a_s,b_s\geq 1)
\end{split}\]
is called the duality 
(here $\{\ldots\}^l$ denotes the $l$ times repetition of 
the sequence in the curly brackets). 
It is also known that some of these relations have counterparts for MZSVs. 
For instance, the sum formula for MZSVs is 
\begin{equation}\label{eq:SF MZSV}
\sum_{\substack{k_1\geq 2,k_2,\ldots,k_n\geq 1,\\ k_1+\cdots+k_n=k}}
\zeta^\star(\bk)=\binom{k-1}{n-1}\zeta(k)\qquad (k>n\geq 1), 
\end{equation}
while no complete counterpart of the duality has not been obtained 
(see \cite{KO,Li,Yz} and \cite{TY} for some results on this problem). 

We prove that some relations among MZVs and MZSVs can be 
extended to those among polynomials $\zeta^t$. 
An example is the sum formula for $\zeta^t$: 
\begin{thm}\label{thm:SF}
For any integers $k>n\geq 1$, we have 
\begin{equation}\label{eq:SF poly}
\sum_{\substack{k_1\geq 2,k_2,\ldots,k_n\geq 1,\\ k_1+\cdots+k_n=k}}
\zeta^t(\bk)
=\Biggl(\sum_{j=0}^{n-1}\binom{k-1}{j}t^j(1-t)^{n-1-j}\Biggr)\zeta(k). 
\end{equation}
\end{thm}

\bigskip 

We study these polynomials from the viewpoint of harmonic algebra, 
the general setup of which was developed in \cite{IKOO}. 
We briefly recall it in Section 2. 
In Section 3, we define an operator $S^t$ which gives 
an algebraic interpretation of \eqref{eq:poly} 
and a variant of the harmonic product compatible with $S^t$, 
and study some basic properties of them. 
Section 4 is devoted to the generalizations of 
the sum formula (Theorem \ref{thm:SF} above) and 
the cyclic sum formula (Theorem \ref{thm:CSF}). 
Finally, in Section 5, we consider the values at $t=\frac{1}{2}$ and 
discuss an application to the two-one conjecture of Ohno-Zudilin \cite{OZ}. 

\subsection*{Acknowledgement} 
The author would like to thank Dr.~Shingo Saito 
for helping to prove Proposition \ref{prop:log S^t}. 

\section{General setup for harmonic algebra}
Here we recall the formulation of harmonic algebra 
introduced in \cite{IKOO}. 

Let $\frA$ be a commutative $\Q$-algebra, 
$\frh^1$ a non-commutative $\frA$-algebra of polynomials of a letter set $A$, 
and $\frz$ the $\frA$-submodule generated by $A$. 

We assume that $\frz$ has a commutative $\frA$-algebra structure, 
not necessarily unitary, with product operation $\circ$ 
called the circle product. 
We define an action of $\frz$ on $\frh^1$ by 
\[a\circ 1=0 \text{ and } a\circ (bw)=(a\circ b)w,\]
where $a,b\in A$ and $w\in\frh^1$ is a word (i.e.\ a monic monomial) 
and by $\frA$-linearity. 

There are two $\frA$-bilinear products $*$ and $\star$ on $\frh^1$ 
defined by 
\begin{gather*}
w*1=1*w=w,\qquad  w\star 1=1\star w=w, \\
(aw)*(bw')=a(w*bw')+b(aw*w')+(a\circ b)(w*w'), \\
(aw)\star(bw')=a(w\star bw')+b(aw\star w')-(a\circ b)(w\star w'),
\end{gather*}
for $a,b\in A$ and any words $w,w'\in\frh^1$. 
With each of these products, $\frh^1$ becomes 
a unitary commutative $\frA$-algebra, 
denoted by $\frh^1_*$ and $\frh^1_\star$ respectively. 

\begin{ex}\label{ex:MZV}
The basic example of the above setting is the case of $\frA=\Q$, 
$A=\{z_k\mid k=1,2,\ldots\}$, and $z_k\circ z_l=z_{k+l}$. 
Then the algebras $\frh^1_*$ and $\frh^1_\star$ give formal expressions 
of the harmonic product of MZVs and MZSVs respectively. 
Namely, set $\frh^0=\frA\oplus\bigoplus_{k=2}^\infty z_k\frh^1$ 
and define $\Q$-linear maps $Z,Z^\star\colon \frh^0\longrightarrow \R$ by 
\begin{align*}
Z\colon 1&\longmapsto 1,\quad 
z_{k_1}\cdots z_{k_n}\longmapsto \zeta(k_1,\ldots,k_n), \\
Z^\star\colon 1&\longmapsto 1,\quad 
z_{k_1}\cdots z_{k_n}\longmapsto \zeta^\star(k_1,\ldots,k_n). 
\end{align*}
Then $\frh^0$ is a $\Q$-subalgebra of $\frh^1$ with respect to 
both products $*$ and $\star$, 
and $Z\colon\frh^0_*\longrightarrow\R$ and 
$Z^\star\colon\frh^0_\star\longrightarrow\R$ are algebra homomorphisms. 
\end{ex}

Another interesting example is the $q$-analogue of the multiple zeta values 
introduced in \cite{Bradley}. 
See \cite[Example 2]{IKOO} for details. 

\section{The interpolating operator and $t$-harmonic product}

\subsection{The operator $S^t$}
\begin{defn}
Let $t$ be an indeterminate, and define an $\frA[t]$-linear operator $S^t$ 
on $\frh^1[t]$ by 
\begin{equation}\label{eq:S^t defn}
S^t(1)=1,\qquad S^t(aw)=aS^t(w)+t\,a\circ S^t(w)
\end{equation}
(where $a\in A$ and $w\in\frh^1$ is a word). 
\end{defn}

Note that, if $\alpha$ is an element of a commutative $\frA$-algebra $\frA'$, 
one naturally obtain an $\frA'$-linear operator $S^\alpha$ 
on $\frA'\otimes_\frA\frh^1$ by substituting $\alpha$ into $t$. 

\begin{ex}
$S^0$ is the identity map on $\frh^1$, 
while $S=S^1$ is just the map considered in \cite{IKOO}, 
which satisfies $Z\circ S=Z^\star$ in the case of Example \ref{ex:MZV}. 
\end{ex}

Here we show some basic properties of $S^t$. 
First, to describe $S^t(w)$ for a word $w$ explicitly, 
we introduce the following notation: 
For an integer $n\geq 0$, denote by $R_n$ the set of subsequences 
$r=(r_0,\ldots,r_s)$ of $(0,\ldots,n)$ such that $r_0=0$ and $r_s=n$. 
For such $r$ and a word $w=a_1\cdots a_n\in\frh^1$, we define the word 
$\Con_r(w)$ (the contraction of $w$ with respect to $r$) by 
\[\Con_r(w)=b_1\cdots b_s,\qquad b_i=a_{r_i+1}\circ\cdots\circ a_{r_{i+1}}. \]
We also put $\sigma(r)=n-s$. Note that $\sigma(r)$ represents 
how many circle products there are in the definition of $\Con_r(w)$. 

\begin{prop}\label{prop:S^t explicit}
For any word $w\in\frh^1$ of length $n$, we have 
\begin{equation*}
S^t(w)=\sum_{r\in R_n}t^{\sigma(r)}\Con_r(w). 
\end{equation*}
\end{prop}
\begin{proof}
It is clear that the above formula determines 
a map $S^t$ satisfying the relation \eqref{eq:S^t defn}. 
\end{proof}

From Proposition \ref{prop:S^t explicit}, we see that 
$(S^t-1)^n(a_1\cdots a_n)=0$ for $n\geq 1$. Therefore, 
\begin{equation}\label{eq:log S^t}
\log S^t=\sum_{k=1}^\infty \frac{(-1)^{k-1}}{k}(S^t-1)^k
\end{equation}
is well-defined as an operator on $\frh^1[t]$. 

\begin{prop}\label{prop:log S^t}
For a word $w$ of length $n$, we have 
\[(\log S^t)(w)=t\sum_{r\in R_n,\,\sigma(r)=1}\Con_r(w)
=t(\log S)(w). \]
\end{prop}
\begin{proof}
The following proof is based on an idea of Shingo Saito. 

By Proposition \ref{prop:S^t explicit}, we can write 
\[(S^t-1)^k(w)=\sum_{r\in R_n}D(\sigma(r),k)t^{\sigma(r)}\Con_r(w), \]
where $D(\sigma(r),k)$ represents the number of all possibilities 
to obtain $\Con_r(w)$ from $w$ by taking non-trivial contractions $k$-times, 
which depends only on $\sigma(r)$ and $k$. In fact, $D(m,k)$ is equal to 
the cardinality of the set 
\[\bigl\{(A_1,\ldots,A_k)\bigm|\{1,\ldots,m\}=A_1\amalg\cdots\amalg A_k, 
A_1,\ldots,A_k\ne\emptyset\bigr\}. \]
By dividing this set into two subsets with respect to whether 
the element $m$ forms a singleton or not, we obtain the recurrence relation 
\[D(m,k)=kD(m-1,k-1)+kD(m-1,k). \]
This relation immediately leads to the formula 
\[\sum_{k=1}^m\frac{(-1)^{k-1}}{k}D(m,k)
=\begin{cases} 1 & (m=1), \\ 0 & (m>1), \end{cases}\]
and then the first equality in the proposition follows. 
The second equality is deduced from the first. 
\end{proof}

\begin{cor}\label{cor:S^t}
\begin{enumerate}
\item We have $S^t=\exp(t\log S)$, where the right hand side 
is defined by the usual infinite series. 
\item We have the equality $S^{t_1+t_2}=S^{t_1}\circ S^{t_2}$ 
as operators on $\frh^1[t_1,t_2]$. 
In particular, $S^{-t}$ is the inverse map of $S^t$. 
\item For any word $w$ of length $n$, we have 
\[\frac{d}{dt}S^t(w)=(S^t\circ\log S)(w)
=\sum_{r\in R_n,\,\sigma(r)=1}S^t(\Con_r(w)). \]
\end{enumerate}
\end{cor}
\begin{proof}
Clear from Proposition \ref{prop:log S^t}. 
\end{proof}

\subsection{The $t$-harmonic product}
We define an $\frA[t]$-bilinear product $\hp{t}$ on $\frh^1[t]$ 
by the recursive formula 
\begin{gather*}
1*w=w*1=1, \\
\begin{split}
(aw)\hp{t}(bw')&=a(w\hp{t}bw')+b(aw\hp{t}w')
+(1-2t)(a\circ b)(w\hp{t}w')\\
&\quad+(t^2-t)a\circ b\circ (w\hp{t}w') 
\end{split}
\end{gather*}
for $a,b\in A$ and words $w,w'\in\frh^1$. 
We call it the $t$-harmonic product. 
Again, we may substitute any element $\alpha$ of 
a commutative $\frA$-algebra $\frA'$ to obtain 
an $\frA'$-bilinear product $\hp{\alpha}$ on $\frA'\otimes_\frA\frh^1$. 
For example, $\hp{0}=*$ and $\hp{1}=\star$. 

\begin{thm}
The product $\hp{t}$ gives a commutative $\frA[t]$-algebra structure 
on $\frh^1[t]$ and the map $S^t$ is an algebra isomorphism of 
$(\frh^1[t],\hp{t})$ onto $(\frh^1[t],*)$. 
\end{thm}
\begin{proof}
Since $S^t$ is $\frA[t]$-linear and bijective, 
it is sufficient to show 
\begin{equation}\label{eq:S multiplicative}
S^t(w_1\hp{t}w_2)=S^t(w_1)*S^t(w_2)\qquad (w_1,w_2\in \frh^1[t]). 
\end{equation}
This is proved in the same way as \cite[Theorem 1]{IKOO}. 
Namely, we prove \eqref{eq:S multiplicative} for words $w_1,w_2$ 
by induction on the length of $w_1w_2$. 
For $a,b\in A$ and words $w_1,w_2$, putting $W_1=S^t(w_1)$ and $W_2=S^t(w_2)$, 
one can show by direct calculation that 
\begin{align*}
S^t(aw_1\hp{t}bw_2)
=&a\bigl(W_1*(bW_2)\bigr)+ta\bigl(W_1*(b\circ W_2)\bigr)\\
&+ta\circ\bigl(W_1*(bW_2)\bigr)+t^2a\circ\bigl(W_1*(b\circ W_2)\bigr)\\
&+b\bigl((aW_1)*W_2\bigr)+tb\bigl((a\circ W_1)*W_2\bigr)\\
&+tb\circ\bigl((aW_1)*W_2\bigr)+t^2b\bigl((a\circ W_1)*W_2\bigr)\\
&+(1-2t)(a\circ b)(W_1*W_2)-t^2a\circ b\circ(W_1*W_2). 
\end{align*}
On the other hand, Lemma 1 and Lemma 2 of \cite{IKOO} state that 
\begin{align*}
&(a\circ W_1)*(bW_2)\\
&\quad=a\circ\bigl(W_1*(bW_2)\bigr)+b\bigl((a\circ W_1)*W_2\bigr)
-(a\circ b)(W_1*W_2), \\
&(aW_1)*(b\circ W_2)\\
&\quad=b\circ\bigl((aW_1)*W_2\bigr)+a\bigl(W_1*(b\circ W_2)\bigr)
-(a\circ b)(W_1*W_2),\\
&(a\circ W_1)*(b\circ W_2)\\
&\quad=a\circ\bigl(W_1*(b\circ W_2)\bigr)+b\circ\bigl((a\circ W_1)*W_2\bigr)
-(a\circ b)\circ(W_1*W_2). 
\end{align*}
Combining these formulas and the definition 
\[(aW_1)*(bW_2)=a\bigl(W_1*(bW_2)\bigr)+b\bigl((aW_1)*W_2\bigr)
+(a\circ b)(W_1*W_2), \]
we obtain 
\begin{align*}
S^t(aw_1\hp{t}bw_2)&=(aW_1)*(bW_2)+t(a\circ W_1)*(bW_2)\\
&\quad+t(aW_1)*(b\circ W_2)+t^2(a\circ W_1)*(b\circ W_2)\\
&=(aW_1+ta\circ W_1)*(bW_2+tb\circ W_2)\\
&=S^t(aw_1)*S^t(bw_2) 
\end{align*}
as desired. 
\end{proof}

\subsection{Vanishing of an alternating sum}
As an application of the theory we have developed, 
we prove a natural generalization of the formula (\cite[Proposition 6]{IKOO}) 
\begin{equation}\label{eq:IKOO_Prop6}
\sum_{k=0}^n(-1)^k(a_1\cdots a_k)*S(a_n\cdots a_{k+1})=0, 
\end{equation}
for $n\geq 1$ and $a_1,\ldots,a_n\in A$. 

\begin{prop}\label{prop:alt_sum}
For $n\geq 1$ and $a_1,\ldots,a_n\in A$, we have 
\begin{equation}\label{eq:alt_sum}
\sum_{k=0}^n(-1)^kS^t(a_1\cdots a_k)*S^{1-t}(a_n\cdots a_{k+1})=0. 
\end{equation}
\end{prop}
\begin{proof}
When $n=1$, \eqref{eq:alt_sum} amounts to the obvious equality 
$S^{1-t}(a)-S^t(a)=a-a=0$. 
Let $n\geq 2$ and assume the formula for $n-1$ holds. 
Since \eqref{eq:alt_sum} holds at $t=0$ by virtue of \eqref{eq:IKOO_Prop6}, 
it suffices to show that the derivative of the left hand side 
with respect to $t$ vanishes. 
By the Leibniz rule, we have 
\begin{align*}
\frac{d}{dt}&\sum_{k=0}^n(-1)^kS^t(a_1\cdots a_k)*S^{1-t}(a_n\cdots a_{k+1})\\
=&\sum_{k=2}^n(-1)^k\sum_{i=1}^{k-1}
S^t\bigl(a_1\cdots(a_i\circ a_{i+1})\cdots a_k\bigr)
*S^{1-t}(a_n\cdots a_{k+1})\\
&-\sum_{k=0}^{n-2}(-1)^k\sum_{i=k+1}^{n-1}
S^t(a_1\cdots a_k)
*S^{1-t}\bigl(a_n\cdots(a_{i+1}\circ a_i)\cdots a_{k+1}\bigr)\\
=&-\sum_{i=1}^{n-1}\Biggl\{
\sum_{k=0}^{i-1}(-1)^k
S^t(a_1\cdots a_k)
*S^{1-t}\bigl(a_n\cdots(a_{i+1}\circ a_i)\cdots a_{k+1}\bigr)\\
&\qquad +\sum_{k=i+1}^n(-1)^{k-1}
S^t\bigl(a_1\cdots(a_i\circ a_{i+1})\cdots a_k\bigr)
*S^{1-t}(a_n\cdots a_{k+1})\Biggr\}. 
\end{align*}
Here the expression inside the bracket vanishes, 
for each $i=1,\ldots,n-1$, by the induction hypothesis applied to the word 
$a_1\cdots(a_i+a_{i+1})\cdots a_n$. This completes the proof. 
\end{proof}

\section{Application to the two-one formula}
In this section, we apply our method to study a conjecture 
proposed in Ohno-Zudilin \cite{OZ}, called the two-one formula. 

Let us be in the setting of Example \ref{ex:MZV}. 
We denote the map $\frh^0[t]\to\R[t]$ obtained from $Z\colon\frh^0\to\R$ 
by tensoring $\Q[t]$ by the same letter $Z$. 
We put $Z^t=Z\circ S^t\colon\frh^0[t]\longrightarrow\R$ and 
$\zeta^t(k_1,\ldots,k_n)=Z^t(z_{k_1}\cdots z_{k_n})$. 
Then $\frh^0[t]$ forms a subalgebra of $(\frh^1[t],\hp{t})$, 
and $Z^t$ is an algebra homomorphism. 

\begin{conj}[Two-one formula]\label{conj:two-one}
For integers $n\geq 1$ and $j_1\geq 1$, $j_2,\ldots,j_n\geq 0$, 
\[\zeta^\star\bigl(\{2\}^{j_1},1,\{2\}^{j_2},1,\ldots,\{2\}^{j_n},1\bigr)
=2^n\zeta^{\frac{1}{2}}(2j_1+1,2j_2+1,\ldots,2j_n+1). \]
\end{conj}

In \cite{OZ}, Ohno and Zudilin proved Conjecture \ref{conj:two-one} for $n=2$, 
and deduced the formula (for $j_1,j_2\geq 1$) 
\[\zeta^\star\bigl(\{2\}^{j_1},1,\{2\}^{j_2},1\bigr)
+\zeta^\star\bigl(\{2\}^{j_2},1,\{2\}^{j_1},1\bigr)\\
=\zeta^\star\bigl(\{2\}^{j_1},1\bigr)\zeta^\star\bigl(\{2\}^{j_2},1\bigr)\]
as a corollary. 
In \cite{TY}, Tasaka and the author obtained a direct proof of this formula 
and its generalization: 

\begin{thm}[{\cite[Theorem 1.2 (i)]{TY}}]
For an integer $n\geq 1$ and non-negative integers $j_1,\ldots,j_n$ 
with $j_1,j_n\geq 1$, we have 
\begin{equation}\label{eq:two-one alt1}
\sum_{k=0}^n(-1)^k
\zeta^\star\bigl(\{2\}^{j_1},1,\ldots,\{2\}^{j_k},1\bigr)
\zeta^\star\bigl(\{2\}^{j_n},1,\ldots,\{2\}^{j_{k+1}},1\bigr)=0. 
\end{equation}
\end{thm}

Therefore, assuming Conjecture \ref{conj:two-one}, 
we should also have 
\begin{equation}\label{eq:two-one alt2}
\sum_{k=0}^n(-1)^kZ^{\frac{1}{2}}(z_{2j_1+1}\cdots z_{2j_k+1})
Z^{\frac{1}{2}}(z_{2j_n+1}\cdots z_{2j_{k+1}+1})=0. 
\end{equation}
In fact, this is an immediate consequence of Proposition \ref{prop:alt_sum}. 
The consistency of the identities \eqref{eq:two-one alt1} and 
\eqref{eq:two-one alt2} give an evidence 
for the validity of Conjecture \ref{conj:two-one} for general $n$. 

\bigskip 

Now we investigate the algebraic structure of the right hand side 
of the two-one formula. 

Put $y_j=2z_{2j+1}$ for $j=0,1,2,\ldots$, and 
let $\frh^{1,\odd}$ be the (non-commutative) subalgebra of $\frh^1$ 
generated by all $y_j$. We also put $\frh^{0,\odd}=\frh^{1,\odd}\cap\frh^0$. 

\begin{prop}\label{prop:two-one RHS}
$\frh^{1,\odd}$ forms a subalgebra of $(\frh^1,\hp{1/2})$. 
Its multiplicative structure is determined recursively by 
\[y_iw\hp{1/2}y_jw'
=y_i(w\hp{1/2}y_jw')+y_j(y_iw\hp{1/2}w')-y_{i+j+1}\circ(w\hp{1/2}w'). \]
Moreover, the $\Q$-linear map $\frh^{0,\odd}\longrightarrow\R$ 
defined by 
\[y_{j_1}\cdots y_{j_n}\longmapsto 
2^n\zeta^{\frac{1}{2}}(2j_1+1,\ldots,2j_n+1)\]
is an algebra homomorphism. 
\end{prop}
\begin{proof}
The first and the second assertions follows directly from 
the definition of the $\frac{1}{2}$-harmonic product. 
\end{proof}

Proposition \ref{prop:two-one RHS} and 
Conjecture \ref{conj:two-one} imply the following conjecture, 
which is interesting in its own right: 

\begin{conj}\label{conj:two-one LHS}
The $\Q$-linear map $X\colon\frh^{0,\odd}\longrightarrow\R$ 
defined by 
\[X(y_{j_1}\cdots y_{j_n})
=\zeta^\star\bigl(\{2\}^{j_1},1,\{2\}^{j_2},1,\ldots,\{2\}^{j_n},1\bigr)\]
is an algebra homomorphism with respect to the $\frac{1}{2}$-harmonic product. 
\end{conj}

\section{Sum formula and Cyclic sum formula}
In this section, we continue to use the setting of Example \ref{ex:MZV} 
We want to prove the sum formula and the cyclic sum formula 
for the polynomials $\zeta^t(\bk)$. 
In addition, we also prove the equivalence of the families of linear relations 
obtained by substituting various numbers into $t$. 
For the latter purpose, we first make an elementary consideration 
on submodules of $\frh^1[t]$ which are closed under differentiation. 

\subsection{Differential submodules of $\frh^1[t]$}
We call a $\Q[t]$-submodule $N$ of $\frh^1[t]$ a differential submodule 
if it is closed under the differentiation with respect to $t$, i.e., 
it satisfies 
\[f(t)\in N\implies \frac{df}{dt}(t)\in N. \]

\begin{lem}\label{lem:DiffSubmod}
Let $N\subset\frh^1[t]$ be a differential submodule, 
$\frA$ a $\Q$-algebra and $\alpha\in\frA$. 
Then the $\frA[t]$-module $\frA\otimes_\Q N$ is generated 
by the $\frA$-submodule 
\[N_\alpha=\bigl\{f(\alpha)\bigm|f(t)\in \frA\otimes_\Q N\bigr\}
\subset\frA\otimes_\Q\frh^1. \]
\end{lem}
\begin{proof}
This is obvious from the expansion 
\[f(t)=\sum_{k=0}^na_k(t-\alpha)^k,\qquad 
a_k=\frac{1}{k!}f^{(k)}(\alpha)\in N_\alpha, \]
which holds for any $f(t)\in \frA\otimes_\Q N$. 
\end{proof}

\subsection{Sum formula}
For integers $k>n\geq 1$, put 
\[x_{k,n}=\sum_{\substack{k_1\geq 2,k_2,\ldots,k_n\geq 1,\\ k_1+\cdots+k_n=k}}
z_{k_1}\cdots z_{k_n}\in \frh^1\]
and 
\[P_{k,n}(t)=\sum_{j=0}^{n-1}\binom{k-1}{j}t^j(1-t)^{n-1-j}\in \Q[t].\]
We call the words appearing in $x_{k,n}$ the admissible words of 
weight $k$ and depth $n$. 

\begin{lem}\label{lem:SF}
For an integer $k\geq 2$, 
let $N^{\mathit{SF}}_k\subset\frh^1[t]$ be the $\Q[t]$-submodule 
generated by the elements 
$S^t(x_{k,n})-P_{k,n}(t)z_k$ for all positive integers $n<k$. 
Then $N^{\mathit{SF}}_k$ is a differential submodule. 
\end{lem}
\begin{proof}
One can deduce from Corollary \ref{cor:S^t} (iii) that 
\[\frac{d}{dt}S^t(x_{k,n})=(k-n)S^t(x_{k,n-1}), \]
since for each admissible word $w$ of weight $k$ and depth $n-1$, 
there are exactly $k-n$ admissible words of weight $k$ and depth $n$ 
such that $\Con_r(w')=w$ for some $r\in R_n$ 
(automatically satisfying $\sigma(r)=1$). 
On the other hand, we have 
\begin{align*}
\frac{d}{dt}P_{n,k}(t)
&=\sum_{j=1}^{n-1}j\binom{k-1}{j}t^{j-1}(1-t)^{n-1-j}\\
&\quad -\sum_{j=0}^{n-2}(n-1-j)\binom{k-1}{j}t^j(1-t)^{n-2-j}\\
&=\sum_{j=0}^{n-2}\Biggl\{(j+1)\binom{k-1}{j+1}+(j+1-n)\binom{k-1}{j}\Biggr\}
t^j(1-t)^{n-2-j}. 
\end{align*}
Since 
\begin{align*}
(j+1)\binom{k-1}{j+1}+(j+1-n)\binom{k-1}{j}
&=(j+1)\binom{k}{j+1}-n\binom{k-1}{j}\\
&=(k-n)\binom{k-1}{j}, 
\end{align*}
we obtain 
\[\frac{d}{dt}P_{k,n}(t)=P_{k,n-1}(t). \]
Hence we have 
\[\frac{d}{dt}\Bigl(S^t(x_{k,n})-P_{k,n}(t)z_k\Bigr)
=S^t(x_{k,n-1})-P_{k,n-1}(t)z_k. \]
This completes the proof. 
\end{proof}

\begin{proof}[Proof of Theorem \ref{thm:SF}]
What we have to prove is that the differential submodule 
$N^{\mathit{SF}}_k\subset\frh^1[t]$ defined in Lemma \ref{lem:SF} 
is contained in $\Ker Z$. 
By Lemma \ref{lem:DiffSubmod}, it suffices to show that the $\Q$-submodule 
$(N^{\mathit{SF}}_k)_0=\bigl\{f(0)\bigm|f(t)\in N^{\mathit{SF}}_k\bigr\}$ 
of $\frh^1$ is contained in $\Ker Z$. 
This is, however, exactly what the sum formula \eqref{eq:SF MZV} 
for usual MZVs says. Hence the proof is complete. 
\end{proof}

\begin{rem}
One sees from the above proof that 
the sum formula \eqref{eq:SF MZSV} for MZSVs follows from 
the sum formula \eqref{eq:SF MZV} for MZVs and vice versa. 
This fact was first proved (essentially) by Hoffman \cite{H}. 
\end{rem}

\subsection{Cyclic sum formula}
Let us recall the cyclic sum formulas for MZVs and MZSVs, 
established respectively by Hoffman-Ohno \cite{HO} and 
Ohno-Wakabayashi \cite{OW}: 
\begin{align}
\label{eq:CSF_MZV}
\sum_{l=1}^n\sum_{j=1}^{k_l-1}
\zeta(k_l+1-j&,k_{l+1},\ldots,k_n,k_1,\ldots,k_{l-1},j)\\
\notag
&=\sum_{l=1}^n\zeta(k_l+1,\ldots,k_n,k_1,\ldots,k_{l-1}), \\
\label{eq:CSF_MZSV}
\sum_{l=1}^n\sum_{j=1}^{k_l-1}
\zeta^\star(k_l+1-j&,k_{l+1},\ldots,k_n,k_1,\ldots,k_{l-1},j)\\
\notag
&=k\zeta(k+1). 
\end{align}
Here $k_1,\ldots,k_n\geq 1$ are not all $1$, and $k=k_1+\cdots+k_n$. 

The cyclic sum formula for the polynomials $\zeta^t$ is the following: 

\begin{thm}\label{thm:CSF}
Let $n\geq 1$ and $k_1,\ldots,k_n\geq 1$ be positive integers, 
and assume that $k_1,\ldots,k_n$ are not all $1$. 
Put $k=k_1+\cdots+k_n$. Then 
\begin{equation}\label{eq:CSF}
\begin{split}
&\sum_{l=1}^n\sum_{j=1}^{k_l-1}
\zeta^t(k_l+1-j,k_{l+1},\ldots,k_n,k_1,\ldots,k_{l-1},j)\\
&=(1-t)\sum_{l=1}^n\zeta^t(k_l+1,\ldots,k_n,k_1,\ldots,k_{l-1})
+t^nk\zeta(k+1). 
\end{split}
\end{equation}
\end{thm}

\smallskip
We introduce some symbols: 
For any word $w=z_{k_1}\cdots z_{k_n}$ of length $n\geq 1$, we put 
\begin{align*}
C(w)&=\sum_{l=1}^n 
z_{k_l+1}z_{k_{l+1}}\cdots z_{k_n}z_{k_1}\cdots z_{k_{l-1}}, \\
\Sigma(w)&=\sum_{l=1}^n\sum_{j=1}^{k_l-1} 
z_{k_l+1-j}z_{k_{l+1}}\cdots z_{k_n}z_{k_1}\cdots z_{k_{l-1}}z_{j}. 
\intertext{Moreover, if $n\geq 2$, we set (putting $k_{n+1}:=k_1$)}
\delta(w)&=\sum_{l=1}^n z_{k_l+k_{l+1}}
z_{k_{l+2}}\cdots z_{k_n}z_{k_1}\cdots z_{k_{l-1}}. 
\end{align*}
We extend them to $\Q[t]$-linear operators. 

\begin{lem}\label{lem:C_Sigma}
For any word $w$ of length $n\geq 1$, we have 
\begin{align*}
\frac{d}{dt}S^t\bigl(C(w)\bigr)
&=\begin{cases}
0 & (n=1),\\
S^t\bigl(C\delta(w)\bigr) & (n\geq 2),
\end{cases} \\
\frac{d}{dt}S^t\bigl(\Sigma(w)\bigr)
&=\begin{cases}
(k-1)z_{k+1} & (n=1,w=z_k), \\
S^t\bigl(\Sigma\delta(w)\bigr)-S^t\bigl(C(w)\bigr) & (n\geq 2). 
\end{cases}
\end{align*}
\end{lem}
\begin{proof}
The case $n=1$ is immediate from the definition. 
Let $n\geq 2$ and write $w=z_{k_1}\cdots z_{k_n}$. 
Then, by Corollary \ref{cor:S^t} (iii), we have 
\begin{align*}
\frac{d}{dt}S^t\bigl(C(w)\bigr)
&=S^t\Biggl(z_1\circ\sum_{l=1}^n\sum_{\substack{r=1\\ r\ne l-1}}^n
z_{k_l}\cdots z_{k_r+k_{r+1}}\cdots z_{k_{l-1}}\Biggr)\\
&=S^t\Biggl(z_1\circ\sum_{r=1}^n\sum_{\substack{l=1\\ l\ne r+1}}^n
z_{k_l}\cdots z_{k_r+k_{r+1}}\cdots z_{k_{l-1}}\Biggr)\\
&=S^t\bigl(C\delta(w)\bigr). 
\end{align*}
Similarly, we have $\frac{d}{dt}S^t\bigl(\Sigma(w)\bigr)=S^t(X)$, 
where 
\begin{align*}
X&=z_1\circ\sum_{l=1}^n\sum_{j=1}^{k_l-1}\Biggl\{
\sum_{\substack{r=1\\ r\ne l,l-1}}^n
z_{k_l-j}\cdots z_{k_r+k_{r+1}}\cdots z_{k_{l-1}}z_j\\
&\hspace{7em}+z_{k_l-j+k_{l+1}}\cdots z_{k_{l-1}}z_j
+z_{k_l-j}\cdots z_{k_{l-1}+j}\Biggr\}\\
&=z_1\circ\sum_{r=1}^n\Biggl\{\sum_{\substack{l=1\\ l\ne r,r+1}}^n
\sum_{j=1}^{k_l-1}z_{k_l-j}\cdots z_{k_r+k_{r+1}}\cdots z_{k_{l-1}}z_j\\
&\hspace{5em}+\sum_{\substack{j=1\\ j\ne k_r}}^{k_r+k_{r+1}-1}
z_{k_r+k_{r+1}-j}z_{k_{r+2}}\cdots z_{k_{r-1}}z_j\Biggr\}\\
&=\Sigma\delta(w)-C(w). 
\end{align*}
This completes the proof of the lemma. 
\end{proof}

\begin{lem}\label{lem:CSF}
For an integer $k\geq 2$, let $N^{\mathit{CSF}}_k\subset\frh^1[t]$ be 
the $\Q[t]$-submodule generated by 
\[f_w(t)=S^t\bigl(\Sigma(w)\bigr)+(t-1)S^t\bigl(C(w)\bigr)-kt^nz_{k+1},\]
where $w=z_{k_1}\cdots z_{k_n}$ runs over all words 
such that $n<k$ and $k=k_1+\cdots+k_n$. 
Then $N^{\mathit{CSF}}_k$ is a differential submodule. 
\end{lem}
\begin{proof}
We extend $f_w(t)$ by linearity on $w$, e.g., 
$f_{w_1+w_2}(t)=f_{w_1}(t)+f_{w_2}(t)$. 
Then Lemma \ref{lem:C_Sigma} implies that 
\[\frac{d}{dt}f_w(t)=\begin{cases}
0 & (n=1),\\
f_{\delta w}(t) & (n\geq 2). 
\end{cases}\]
Thus we obtain the lemma. 
\end{proof}

\begin{proof}[Proof of Theorem \ref{thm:CSF}]
Now the proof proceeds in the same way as the proof of Theorem \ref{thm:SF}. 
Namely, by Lemma \ref{lem:CSF} and Lemma \ref{lem:DiffSubmod}, 
it is sufficient to show that 
$(N^{\mathit{CSF}}_k)_0=\bigl\{f(0)\bigm|f(t)\in N^{\mathit{CSF}}_k\bigr\}$ 
of $\frh^1$ is contained in $\Ker Z$, 
and this is the content of the cyclic sum formula \eqref{eq:CSF_MZV} 
for usual MZVs. 
\end{proof}

\begin{rem}
The above proof implies that 
the formulas \eqref{eq:SF MZV} and \eqref{eq:SF MZSV} are equivalent, 
as proved in \cite{IKOO} and \cite{TW}. 
\end{rem}

\end{document}